\documentclass{amsart}
\usepackage{amsmath,amsthm,amssymb,amsfonts}
\usepackage[all]{xy}
\usepackage[hmargin = 2.5cm]{geometry}
\usepackage{tikz}
\usetikzlibrary{arrows}

\def\modd{{\textrm{mod-}}}

\def\k{{\mathbf k}}
\def\kQ{\k Q}

\def\X{{\mathbf X}}
\def\x{{\mathbf x}}

\def\uS{\underline \Sigma}
\def\uA{\mathcal A_{\uS}}
\def\uchi{{\mathcal X}_{\uS}}
\def\uf{\underline f}

\def\qpm{q^{\pm \frac 12}}
\def\qp{q^{\frac 12}}
\def\qm{q^{-\frac 12}}

\def\Z{{\mathbb{Z}}}

\def\fl{{\longrightarrow}\,}
\def\ens#1{\left\{ #1 \right\}}
\def\<{\left<}
\def\>{\right>}

\newtheorem{theorem}{Theorem}[section]
\newtheorem{prop}[theorem]{Proposition}
\newtheorem{lem}[theorem]{Lemma}
\newtheorem{corol}[theorem]{Corollary}

\theoremstyle{definition}
\newtheorem{defi}[theorem]{Definition}

\newtheorem{exmp}[theorem]{Example}

\title{Quantum frieze patterns in quantum cluster algebras of type $A$}
\author{Jean-Philippe Burelle and Gr\'egoire Dupont}
\email{
Jean-Philippe.Burelle@USherbrooke.ca
}
\email{
dupontg@math.jussieu.fr
}
\address{
Universit\'e de Sherbrooke, 2500 Boul. de l'universit\'e, J1K 2R1 Sherbrooke QC, Canada.
}
\address{
Institut de Math\'ematiques de Jussieu -- Paris Rive Gauche\\
Universit\'e Denis Diderot -- Paris 7\\
175 rue du chevaleret\\
75013 Paris, France.}

\begin{document}

\begin{abstract}
	We introduce a quantisation of the Coxeter-Conway frieze patterns and prove that they realise quantum cluster variables in quantum cluster algebras associated with linearly oriented Dynkin quivers of type $A$. As an application, we obtain the explicit polynomials arising from the lower bound phenomenon in these quantum cluster algebras. 
\end{abstract}

\maketitle

\section{Introduction, notations and background}
	\subsection{Introduction}
		Quantum cluster algebras were introduced by Berenstein and Zelevinsky as quantum deformations of cluster algebras \cite{BZ:quantum}. The aim of introducing quantum cluster algebras was to provide a \emph{combinatorial} model for studying dual canonical bases in quantum groups and several recent articles showed the relevance of these structures with respect to this goal, see for instance \cite{Lampe:Kronecker, Kimura:qcanonical, Lampe:typeA, GLS:qcluster, HL:qk0}. 

		The simplest (quantum or classical) cluster algebras are those associated to Dynkin quivers of type $A$ which, besides their simplicity, carry very nice combinatorial structures and play an important role with respect to certain quantum groups associated to $\mathfrak {sl}_{n+1}$, see for instance \cite{Lampe:typeA}. 

		One of the various combinatorial techniques for studying the classical cluster algebras of type $A$ are \emph{frieze patterns}, which first appeared in the early 70's in an article by Coxeter, see \cite{C71}. In the context of cluster algebras, frieze patterns first arose in \cite[Section 5]{CC} where it turned out that they mimic the mutations along sinks or sources in a cluster algebra of type $A$. More generally, as it was observed in \cite{ARS:frises}, friezes can be used in order to realise certain cluster variables in arbitrary acyclic cluster algebras. These friezes were then studied in numerous articles, see for instance \cite{Propp:frieze,Dupont:stabletubes,BM:friezesD,AD:algorithm,ADSS:strings}, and they found interesting applications, in particular with respect to positivity, see \cite{ARS:frises,ADSS:strings,Dupont:positiveregular}.

		The aim of this article is to define a quantum analogue of Coxeter's frieze patterns and to show that they play a similar role in the context of quantum cluster algebras of type $A$ as the (classical) frieze patterns do in the context of (classical) cluster algebras of type $A$.

		As an application of these quantum frieze patterns, we obtain the explicit polynomials arising from the lower bound phenomena in quantum cluster algebras associated with a linearly oriented quiver of type $A$. These polynomials are quantum deformations of the so-called \emph{generalised Chebyshev polynomials} which appeared in the context of classical cluster algebras in \cite{Dupont:stabletubes}, or in the context of $SL(2,\Z)$-tilings of the plane under the name of \emph{signed continuant polynomials} in \cite{BR:Slktilings}. These polynomials turn out to be completely different from another deformation of the generalised Chebyshev polynomials, which was introduced in \cite{Dupont:qChebyshev} under the name of \emph{quantized Chebyshev polynomials}. In order not to get confused by this unfortunate choice of terminology, we refer to the polynomials introduced in the present article as \emph{quantum signed continuant polynomials}.

	\subsection{Notations for quantum cluster algebras}
		We shall make free use of the standard notations and terminology from the theory of quantum cluster algebras. Our reference is \cite{BZ:quantum}.

		Throughout the article, $n$ is a positive integer, $Q$ is the quiver with vertices indexed by $Q_0=\ens{1, \ldots, n}$ and arrows $i \fl i+1$ for any $1 \leq i \leq n-1$ and $B =(b_{i,j})_{1 \leq i,j \leq n} \in M_n(\Z)$ is the incidence matrix of $Q$, that is the skew-symmetric matrix given by 
		$$b_{i,j} = \left\{\begin{array}{ll}
			1 & \text{if } j=i+1~;\\
			-1 & \text{if } j=i-1~;\\
			0 & \text{otherwise.}
		\end{array}\right.$$
		Since we want to embed $B$ into a quantum seed, we need it to be invertible and therefore we assume throughout the article that $n$ is an \emph{even} integer.

		We set $\Lambda=(B^t)^{-1} = (\Lambda_{i,j})_{1 \leq i,j \leq n} \in M_n(\Z)$ the matrix given by 
		$$\Lambda_{i,j} = \left\{\begin{array}{ll}
			1 & \text{if $i$ is odd, $j$ is even and } i<j~;\\
			-1 & \text{if $i$ is even, $j$ is odd and } i>j~;\\
			0 & \text{otherwise.}
		\end{array}\right.$$
		Note that we shall usually identify $\Lambda$ with the bilinear form over $\Z^n$ with matrix $\Lambda$ when expressed in the canonical basis $\ens{e_1, \ldots, e_n}$ of $\Z^n$.

		We fix an indeterminate $\nu$ over $\Z$ and we set $q=\nu^2$. Note that we usually write $\qp$ instead of $\nu$ and $\qm$ instead of $\nu^{-1}$. We denote by $\mathcal T(\Lambda)$ the \emph{based quantum torus}, that is the $\Z[\qpm]$-algebra with a distinguished $\Z[\qpm]$-basis $\ens{X^u \ | \ u \in \Z^n}$ and with multiplication
		$$X^uX^v = q^{\frac{\Lambda(u,v)}{2}} X^{u+v}.$$
		For any $1 \leq i \leq n$, we set $X_i = X^{e_i}$ and we finally set $\X = (X_1, \ldots, X_n)$. Therefore, $\Sigma = (B, \Lambda, \X)$ is a \emph{quantum seed} in the sense of \cite{BZ:quantum}. Note that, up to a scalar multiple, the above choice of $\Lambda$ is the unique one such that $(B,\Lambda)$ is a compatible pair. Since $\mathcal T(\Lambda)$ is an Ore domain, it is contained in its skew-field of fractions which we denote by $\mathcal F(\Lambda)$. 

		From now on, $\mathcal A_\Sigma$ denotes the \emph{quantum cluster algebra} associated with the quantum seed $\Sigma$. It is the $\Z[\qpm]$-subalgebra of $\mathcal F(\Lambda)$ generated by the so-called \emph{quantum cluster variables}, that is, the elements in $\mathcal F(\Lambda)$ which are obtained from the initial cluster $\X$ by a finite number of mutations.

	\subsection{Notations for additive categorifications}
		Even if our results can be stated in purely combinatorial terms, we found it more enlightening to sometimes refer to the classical background of additive categorifications of acyclic cluster algebras via cluster categories. We refer the reader to \cite{Keller:survey} for an overview of these categorifications but the results should nevertheless be readable independently on this background. 

		In this context, we shall use the following notations~: $\k$ is an algebraically closed field, $\kQ$ is the path algebra of $Q$ over $\k$ and $\modd\kQ$ is the category of finitely generated right $\k Q$-modules. The bounded derived category $D^b(\modd \kQ)$ of $\modd\kQ$ is a triangulated category with suspension $[1]$ and the composition $F=\tau^{-1}[1]$ of the inverse of the Auslander-Reiten translation with the shift functor is an autoequivalence of $D^b(\modd \kQ)$. The orbit category of $F$ in $D^b(\modd\kQ)$ is called the \emph{cluster category} of $Q$ and was first introduced in \cite{BMRRT} (an alternative definition was independently given in \cite{CCS1} for Dynkin quivers of type $A$). 

	\subsection{Organisation of the article}
		The article is organised as follows. Section \ref{section:qfriezes} introduces the notion of quantum frieze pattern. Section \ref{section:qcluster} proves that quantum cluster variables in a cluster algebra of type $A$ satisfy the quantum frieze relations. Finally, Section \ref{section:qCheb} exhibits the explicit polynomials arising from the lower bound phenomena in the quantum cluster algebras of type $A$.
		
\section{Quantum frieze patterns}\label{section:qfriezes}
	In this section, we introduce quantum frieze patterns which, in simple words, can be viewed as tilings of the plane by $2 \times 2$ matrices with quantum determinant 1 over the skew field $\mathcal F(\Lambda)$.

	We denote by $\Z Q$ the \emph{repetition quiver} of $Q$, that is, the quiver whose vertices are indexed by $Q_0 \times \Z$ and with arrows $(i,k) \fl (j,l)$ whenever $k=l$ and $i \fl j$ in $Q$ or $l=k+1$ and $j \fl i$ in $Q$. 
	\begin{exmp}
		The following figure depicts $\Z Q$ when $Q$ is a linearly oriented Dynkin quiver of type $A$ with four vertices.
		
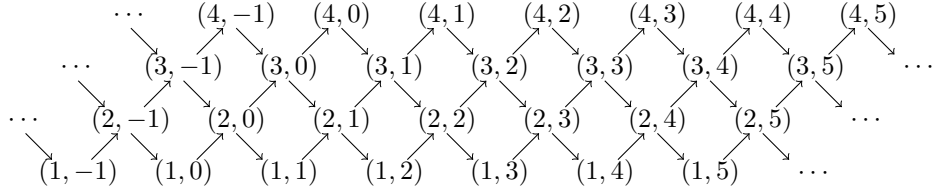
\begin{center}
			\begin{tikzpicture}[scale = .7]
				\foreach \x in {-1,0,...,5}
				{
					\foreach \y in {1,...,3}
					{
						\fill (2*\x+\y-1,\y-1) node {$(\y,\x)$};
						\draw[->] (2*\x+\y-1+.25,\y-1+.25) -- (2*\x+\y-1+.75,\y-1+.75);
						\draw[->] (2*\x+\y,\y-.25) -- (2*\x+\y+.5,\y-.75);
					}
					\foreach \y in {4}
					{
						\fill (2*\x+\y-1,\y-1) node {$(\y,\x)$};
					}
				}
				\foreach \y in {1,...,3}
				{
					\fill (2*5+\y+1,\y-1) node {$\cdots$};
				}
				\foreach \y in {1,...,3}
				{
					\fill (-4+\y,\y) node {$\cdots$};
					\draw[->] (-4+\y,\y-.25) -- (-4+\y+.5,\y-.75);
				}
			\end{tikzpicture}
		\end{center}
	\end{exmp}

	\begin{defi}
		A \emph{quantum frieze pattern on $\Z Q$} is a map $$f : Q_0 \times \Z \fl \mathcal F(\Lambda)$$ such that 
		$$f(i,j)f(i,j+1) - \qp f(i-1,j+1)f(i+1,j) = 1$$
		for any $(i,j) \in Q_0 \times \Z$ with the convention that $f(n+1,j)=1$ and $f(0,j)=1$ for any $j \in \Z$.
	\end{defi}

	If $a,b,c,d \in \mathcal F(\Lambda)$, we use the classical notations from quantum determinants~:
	$$\left|\begin{array}{cc}
		a & b \\ 
		c & d 
	\end{array}\right|_{\nu} = ad - \qp bc.$$
	Therefore, a quantum frieze pattern on $\Z Q$ is a map which satisfies the \emph{quantum unimodular rule}
	$$\left|\begin{array}{cc}
		f(i,j) & f(i-1,j+1) \\ 
		 f(i+1,j) & f(i,j+1) 
	\end{array}\right|_{\nu} = 1$$
	for any $(i,j) \in Q_0 \times \Z$ and can thus be viewed as a \emph{$SL_\nu(2,\mathcal F(\Lambda))$-tiling} of the plane where $SL_\nu(2,\mathcal F(\Lambda))$ denotes the set of $2 \times 2$ matrices over $\mathcal F(\Lambda)$ with quantum determinant 1.

	We recall that $\Z Q$ is isomorphic to the \emph{Auslander-Reiten quiver} of the category $D^b(\modd \kQ)$, see \cite{Happel:book}. The autoequivalence $F$ on $D^b(\modd \kQ)$ induces an automorphism $\phi$ of the quiver $\kQ$ sending $(i,j)$ to $(n-i+1,j+i+1)$. We denote by $\Gamma=\Z Q/\<\phi\>$ the quotient quiver, which is isomorphic to the Auslander-Reiten quiver of the cluster category of $Q$. Therefore, the set $\Gamma_0$ of vertices of $\Gamma$ is identified with a fundamental domain for the action of $\phi$ on the vertices of $\Z Q$ and we chose $\ens{(i,j) \ | \ j \geq 0, i+j \leq n+1}$ as such a fundamental domain.

	\begin{exmp}
		The following figure depicts $\Gamma$ when $Q$ is a linearly oriented quiver with four vertices. The grey zone corresponds to the choice of a certain fundamental domain for the action of the automorphism $\phi$.
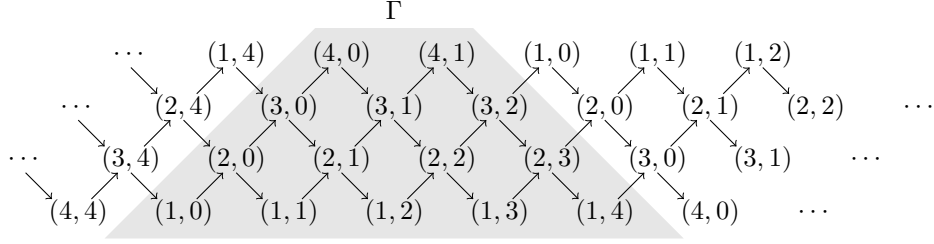
		\begin{center}
			\begin{tikzpicture}[scale = .7]

				\fill (4,3.5) node [above] {$\Gamma$};
				\fill[gray!20] (-1.5,-.5) -- (2.5,3.5) -- (5.5,3.5) -- (9.5,-.5) -- cycle;

				\foreach \x in {-1,...,4}
				{
					\foreach \y in {1,...,3}
					{
						\draw[->] (2*\x+\y-1+.25,\y-1+.25) -- (2*\x+\y-1+.75,\y-1+.75);
						\draw[->] (2*\x+\y,\y-.25) -- (2*\x+\y+.5,\y-.75);
					}

				}
				\foreach \y in {1,...,3}
				{
					\fill (2*5+\y+1,\y-1) node {$\cdots$};
				}
				\foreach \y in {1,...,3}
				{
					\fill (-4+\y,\y) node {$\cdots$};
					\draw[->] (-4+\y,\y-.25) -- (-4+\y+.5,\y-.75);
				}

				\foreach \x in {0,1}
				{
					\foreach \y in {1,...,4}
					{
						\fill (2*\x+\y-1,\y-1) node {$(\y,\x)$};
					}
				}

				\foreach \x in {2}
				{
					\foreach \y in {1,...,3}
					{
						\fill (2*\x+\y-1,\y-1) node {$(\y,\x)$};
					}
				}

				\foreach \x in {3}
				{
					\foreach \y in {1,...,2}
					{
						\fill (2*\x+\y-1,\y-1) node {$(\y,\x)$};
					}
				}

				\foreach \x in {4}
				{
					\foreach \y in {1}
					{
						\fill (2*\x+\y-1,\y-1) node {$(\y,\x)$};
					}
				}

				\fill (7,3) node {$(1,0)$};
				\fill (8,2) node {$(2,0)$};
				\fill (9,1) node {$(3,0)$};
				\fill (10,0) node {$(4,0)$};

				\fill (9,3) node {$(1,1)$};
				\fill (10,2) node {$(2,1)$};
				\fill (11,1) node {$(3,1)$};

				\fill (11,3) node {$(1,2)$};
				\fill (12,2) node {$(2,2)$};

				\fill (1,3) node {$(1,4)$};
				\fill (0,2) node {$(2,4)$};
				\fill (-1,1) node {$(3,4)$};
				\fill (-2,0) node {$(4,4)$};

			\end{tikzpicture}
		\end{center}
	\end{exmp}

	\begin{defi}
		We say that a quantum frieze pattern on $\Z Q$ induces a \emph{quantum frieze pattern on $\Gamma$} if $f(\phi.(i,j)) = f(i,j)$ for any $(i,j) \in Q_0 \times \Z$.
	\end{defi}

	Quantum frieze patterns can easily be computed recursively with particular choices of values. For instance, we have the following lemmas which are proved with a straightforward induction~:
	\begin{lem}\label{lem:slice}
		A quantum frieze pattern on $\Z Q$ (or on $\Gamma$) is entirely determined by its values on the set $\ens{(i,0) \ | \ i \in Q_0}$. \hfill \qed
	\end{lem}

	\begin{lem}\label{lem:mouth}
		A quantum frieze pattern on $\Gamma$ is entirely determined by its values on the set $\ens{(1,j) \ | \ 0 \leq j \leq n}$. \hfill \qed
	\end{lem}

\section{Quantum frieze patterns and quantum cluster variables}\label{section:qcluster}
	In this section, we prove that quantum friezes on $\Gamma$ realise quantum cluster variables in $\mathcal A_\Sigma$. 

	According to Lemma \ref{lem:slice}, there exists a unique quantum frieze pattern on $\Z Q$ such that $f(i,0)=X_i$ for any $i \in Q_0$. This quantum frieze pattern is referred to as the \emph{quantum frieze of variables} associated with $Q$.

	\begin{exmp}
		The following figure depicts the quantum frieze of variables associated with a linearly oriented quiver of Dynkin type $A$ with 2 vertices. 
		
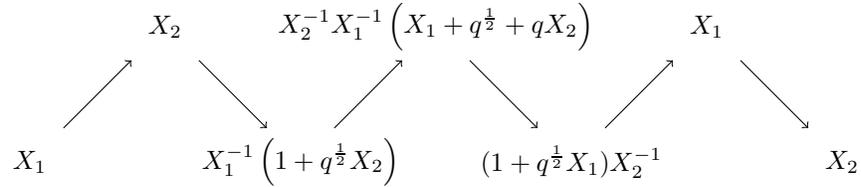
\begin{center}
			\begin{tikzpicture}[scale = .9]
				\fill (0,0) node {$X_1$};

				\draw[->] (.5,.5) -- (1.5,1.5);

				\fill (2,2) node {$X_2$};

				\draw[->] (2.5,1.5) -- (3.5,.5);

				\fill (4,0) node {$X_1^{-1}\left(1+\qp X_2 \right)$};

				\draw[->] (4.5,.5) -- (5.5,1.5);

				\fill (6,2) node {$X_2^{-1}X_1^{-1}\left(X_1+\qp+qX_2\right)$};

				\draw[->] (6.5,1.5) -- (7.5,.5);

				\fill (8,0) node {$(1+\qp X_1)X_2^{-1}$};

				\draw[->] (8.5,.5) -- (9.5,1.5);

				\fill (10,2) node {$X_1$};

				\draw[->] (10.5,1.5) -- (11.5,.5);

				\fill (12,0) node {$X_2$};
			\end{tikzpicture}
		\end{center}
	\end{exmp}

	\begin{theorem}\label{thm:bij}
		Let $f$ be the quantum frieze of variables associated with $Q$. Then $f$ induces a quantum frieze pattern on $\Gamma$ which realises a bijection from $\Gamma_0$ to the set of quantum cluster variables in $\mathcal A_\Sigma$.
	\end{theorem}
	\begin{proof}
		We denote by ${\mathcal X}_\Sigma$ the set of quantum cluster variables in $\mathcal A_\Sigma$.

		We first prove that $f(i,j) \in {\mathcal X}_\Sigma$ for any $(i,j) \in Q_0 \times \Z$. We set $\Sigma^1=(B,\Lambda,\X)$ and for any $1\leq k\leq n$, we set $\Sigma^{k+1}=(B^{k+1},\Lambda^{k+1},\X^{k+1})=\mu_k(\Sigma^{k})$. Let $X'_1, \ldots, X'_{k-1}$ such that $\X^k=(X'_1,X'_2,\cdots ,X'_{k-1},X_{k},\cdots ,X_n)$. Note that the $k$-th column of $B^k$ only contains non-negative numbers and that mutating $B^k$ in the direction $k$ only changes the sign of the $k$-th column. The same holds for $\Lambda^k$.

		We now prove by induction on $i$ that $X_i' = f(i,1)$ for any $1 \leq i \leq n$. Performing the above sequence of mutations, we get the quantum cluster variables~:
		\begin{align*}
		X_1' &= X^{-e_1} + X^{-e_1+e_2}\\
					&= X_1^{-1} + q^{1/2\Lambda_{1,2}}X_1^{-1}X_2\\
					&= X_1^{-1} + q^{1/2}X_1^{-1}X_2,\\
		\end{align*}
		so that 
		$$X_1X_1' = q^{1/2}X_2 + 1,$$
		that is,
		$$f(1,0)X_1' = q^{1/2}f(2,0)+1$$
		and thus $X_1'=f(1,1)$.

		We now fix $2 \leq i \leq n$. Then,
		\begin{align*}
		X_i' &= X^{-e_i} + X^{-e_i+e_{i-1}+e_{i+1}}\\
					&= X_i^{-1} + q^{(\Lambda^i_{i,i-1}+\Lambda^i_{i,i+1}-\Lambda^i_{i-1,i+1})/2}X_i^{-1}X_{i-1}X_{i+1}.\\
		\end{align*}
		We know that $\Lambda^i$ differs from $\Lambda$ only by sign changes. Therefore, $\Lambda^i_{i-1,i+1}=0$. Moreover, as we did not mutate along $i$ or $i+1$ yet, we have $\Lambda^i_{i,i+1}=\Lambda_{i,i+1}$. And as we performed a single mutation in the direction $i-1$, we have $\Lambda^i_{i,i-1}=-\Lambda_{i,i-1}$.
		Thus,
		\begin{align*}
		X_i' &= X_i^{-1} + q^{(-\Lambda_{i,i-1}+\Lambda^i_{i,i+1})/2}X_i^{-1}X_{i-1}X_{i+1}\\
			&= X_i^{-1} + q^{1/2}X_i^{-1}X_{i-1}X_{i+1},
		\end{align*}
		and thus
		$$X_iX_i'= q^{1/2}X'_{i-1}X_{i+1} + 1.$$
		Therefore, by induction hypothesis, we get
		$$f(i,0)X_i'= q^{1/2}f(i-1,1)f(i+1,0) + 1$$
		and thus $X_i' = f(i,1)$ for any $1 \leq i \leq n$.

		We now observe that $B^{n+1}=B$ and $\Lambda^{n+1}=\Lambda$, so that we can again perform the same sequence of mutations and, by the same argument, the newly obtained quantum cluster variables satisfy the quantum frieze relations. Therefore, $f(i,j) \in {\mathcal X}_\Sigma$ for any $(i,j) \in Q_0 \times \Z_{\geq 0}$. The same argument works ``backwards'' performing the sequence $\mu_1 \circ \cdots \circ \mu_n$ of mutations along the sinks and thus $f(i,j) \in {\mathcal X}_\Sigma$ for any $(i,j) \in Q_0 \times \Z$. 

		We now fix a $n$-tuple $\x=(x_1, \ldots, x_n)$ of indeterminates over $\Z$ and we denote by $\uS$ the (classical) seed $(B,\x)$. The corresponding (classical) cluster algebra is denoted by $\uA$ and the set of cluster variables in $\uA$ is denoted by $\uchi$. We denote by $\sigma$ the \emph{specialisation} map from $\mathcal F(\Lambda)$ to $\mathbb Q(x_1, \ldots, x_n)$, which is the $\Z$-algebra homomorphism sending $q$ to 1 and $X_i$ to $x_i$ for any $1 \leq i \leq n$. Finally, we denote by $\uf$ the classical frieze pattern associated with $Q$, that is, the map $\uf:Q_0 \times \Z \fl \mathbb Q(x_1, \ldots, x_n)$ such that $\uf(i,0)=x_i$ for any $1 \leq i \leq n$ and such that $$\uf(i,j)\uf(i,j+1) - \uf(i-1,j+1)\uf(i+1,j) = 1$$
		for any $(i,j) \in Q_0 \times \Z$ with the convention that $\uf(n+1,j)=1$ and $\uf(0,j)=1$ for any $j \in \Z$. It is well-known (see for instance \cite{ARS:frises}) that $\uf$ is $\phi$-periodic and that it induces a bijection between the vertices in $\Gamma_0$ and $\uchi$. 

		We now observe that $\uf(i,j) = \sigma(f(i,j))$ for any $(i,j) \in Q_0 \times \Z$. Since Berenstein and Zelevinsky proved in \cite{BZ:quantum} that $\sigma$ induces a bijection between ${\mathcal X}_\Sigma$ and $\uchi$, it follows that $f$ is also $\phi$-periodic and induces a quantum frieze pattern on $\Gamma$. Moreover, we have the commutative diagram 
		$$\xymatrix{
			& \Gamma_0 \ar[rd]^{\uf} \ar[dl]_{f} \\
			f(\Gamma_0) \ar[d]^{\iota} \ar[rr]^{\sigma} && \uf(\Gamma_0) \ar@{=}[d]\\
			{\mathcal X}_\Sigma \ar[rr]^{\sigma}_{\sim} && \uchi.
		}$$
		where $\iota$ is the inclusion given by the above discussion. As $\uf = \sigma \circ f$ is bijective, $\sigma : f(\Gamma_0) \rightarrow \uf(\Gamma_0)$ is surjective and $f$ is injective. And since $\sigma : {\mathcal X}_\Sigma \rightarrow \uchi$ is a bijection, we obtain
		\[ \# ({\mathcal X}_\Sigma) = \# (\uchi) = \# (\uf(\Gamma_0)) \leq \# (f(\Gamma_0)) \leq \# ({\mathcal X}_\Sigma) .\]
		Thus,
		\[ \# ({\mathcal X}_\Sigma) = \# (f(\Gamma_0))\]
		and therefore $f(\Gamma_0) = {\mathcal X}_\Sigma$, that is, $f$ induces a bijection between $\Gamma_0$ and ${\mathcal X}_\Sigma$.
	\end{proof}

\section{Lower bounds and quantum signed continuant polynomials}\label{section:qCheb}
	Lower bounds for classical cluster algebras were initially introduced in \cite{cluster3} and an analogue was defined in the quantum settings in \cite[Section 7]{BZ:quantum}. For any $1 \leq i \leq n$, we denote by $X_i'$ the quantum cluster variable obtained by mutating the initial cluster $\X$ in the direction $i$ (note that the notation $X_i'$ differs from the one used in the proof of Theorem \ref{thm:bij}). The lower bound of $\mathcal A_\Sigma$ is the $\Z[\qpm]$-subalgebra of $\mathcal F(\Lambda)$~:
	$$\mathcal L_\Sigma = \Z[\qpm][X_1, X_1', \ldots, X_n, X_n'].$$
	According to \cite[Theorem 7.5]{cluster3}, since $Q$ is acyclic, we have $\mathcal A_\Sigma = \mathcal L_\Sigma$ and therefore, every quantum cluster variable in $\mathcal A_\Sigma$ can be expressed as a polynomial in the quantum cluster variables $X_1, X_1', \ldots, X_n, X_n'$. Explicit formulae for these polynomials in type $A$ were for instance obtained in classical settings, see \cite{ST:unpunctured} or \cite{Dupont:stabletubes}. 

	We define a family of polynomials $\ens{P_{m,i}}_{m \geq 0}$ where, for any $1 \leq i \leq m$ and $m \geq 1$ such that $m+i-1 \leq n$, $P_{m,i}=P_m(X'_i,\cdots,X'_{i+m-1})$ is defined by induction, setting $P_{0,i}=1$, $P_{1,i}=X'_i$ and 
	\begin{equation}
		P_{m+1,i}=
		\begin{cases}
			P_{m,i}X'_{i+m} - q^{-\frac{1}{2}}P_{m-1,i} & \textrm{if } m \textrm{ is even}~;\\
			q^{-\frac{1}{2}}(P_{m,i} X'_{i+m} - P_{m-1,i}) & \textrm{if } m \textrm{ is odd.} \\
		\end{cases}\label{eq:Pndefn}
	\end{equation}
	Note that \eqref{eq:Pndefn} is a quantisation of the three-term-relation for the so-called \emph{generalised Chebyshev polynomials} or \emph{signed continuant polynomials}, see \cite[Lemma 3.2]{Dupont:stabletubes} and \cite[(37)]{BR:Slktilings}. 

	The main result of this section is the following~:
	\begin{theorem}\label{thm:qCheb}
		Let $f$ be the quantum frieze of variables associated with $Q$. Then for any $(i,j) \in \Gamma_0$, the quantum cluster variable in $\mathcal A_\Sigma$ corresponding to $(i,j)$ is 
		$$f(i,j) = 
			\begin{cases} 
				X_i & \text{if } j=0~;\\
				P_{i}(X_j', \ldots, X_{j+i-1}') & \text{ otherwise}.
			\end{cases}$$
	\end{theorem}
	The rest of the section is devoted to the proof of the theorem.

	We start with some technical lemmas~:
	\begin{lem}
		For any $1 \leq i < n$, we have~:
		$$ X'_i X'_{i+1} - 1 = q(X'_{i+1}X'_i - 1) $$
		and
		$$ X'_i X'_{i+k} = q^{(-1)^{k-1}} X'_{i+k}X'_i $$
		for any $1< k \leq n-i$. 
	\end{lem}
	\begin{proof}
		For any $1 \leq i \leq n$, we have 
		\begin{align*}
			X'_i X'_{i+1} - 1 &= (X^{-e_i + e_{i-1}} + X^{-e_i + e_{i+1}})(X^{-e_{i+1}+e_i} + X^{-e_{i+1} + e_{i+2}}) - 1 \\
					&= qX^{-e_{i+1}+e_i}X^{-e_i+e_{i-1}} + qX^{-e_{i+1}+e_{i+2}}X^{-e_i+e_{i-1}}\\
					&+ 1 + qX^{-e_{i+1}+e_{i+2}}X^{-e_i+e_{i+1}} - 1 \\
					&= q\left( (X^{-e_{i+1}+e_i} + X^{-e_{i+1}+e_{i+2}})(X^{-e_i+e_{i-1}} + X^{-e_i + e_{i+1}}) \right)\\ 
					&- qX^{-e_{i+1}+e_i}X^{-e_i + e_{i+1}} \\
					&= q(X'_{i+1}X'_i - 1).
		\end{align*}
		For any $1 \leq i < n$ and any $1 < k \leq n-i$, we have 
		\begin{align*}
			X'_i X'_{i+k} &= (X^{-e_i + e_{i-1}} + X^{-e_i + e_{i+1}})(X^{-e_{i+k}+e_{i+k-1}} + X^{-e_{i+k} + e_{i+k+1}}) \\
				&= q^{-1^{k-1}}X^{-e_{i+k}+e_{i+k-1}}X^{-e_i+e_{i-1}} + qX^{-e_{i+k}+e_{i+k+1}}X^{-e_i+e_{i-1}}\\
				&	+ qX^{-e_{i+k}+e_{i+k-1}}X^{-e_i+e_{i+1}} + qX^{-e_{i+k}+e_{i+k+1}}X^{-e_i+e_{i+1}} \\
				&= q\left( (X^{-e_{i+k}+e_{i+k-1}} + X^{-e_{i+k}+e_{i+k+1}})(X^{-e_i+e_{i-1}} + X^{-e_i + e_{i+1}}) \right)\\ 
				&= q(X'_{i+k}X'_i).
		\end{align*}
	\end{proof}

	\begin{lem}\label{lem:pnxn}
		The following hold for $k>1$ and $1\leq i \leq n-(m-1+k)$~:
		$$P_{m,i}X'_{i+m-1+k} = q^{(-1)^{k-1}}X'_{i+m-1+k}P_{m,i} \quad \textrm{if $m$ is odd~;}$$
		$$P_{m,i}X'_{i+m-1+k} = X'_{i+m-1+k}P_{m,i} \quad \textrm{if $m$ is even.}$$
	\end{lem}
	\begin{proof}
		It is enough to prove the statement for the particular case where $i=1$ and therefore, in order to simplify notations, we set $P_m = P_{m,1}$.
		We have
		$$P_1X'_{k+1} = X'_1 X'_{k+1} = q^{(-1)^{k-1}}X'_{k+1}X'_1 = q^{(-1)^{k-1}}X'_{k+1}P_1$$
		and
		\begin{align*}
		P_2X'_{k+1} &= q^{-\frac{1}{2}}(X'_1X'_2-1)X'_{k+1}\\
			&= q^{-\frac{1}{2}}(q^{(-1)^{k-1}+(-1)^{k-2}}X'_{k+1}X'_1X'_2 - X'_{k+1})\\
			&= X'_{k+1}P_2.
		\end{align*}
	
		We assume that the statement is true for $P_m$ and $P_{m+1}$. Then, if $m+2$ is odd, we have
		\begin{align*}
		P_{m+2}X'_{m+2+k} &= (P_{m+1}X'_{m+2} - q^{-\frac{1}{2}}P_m)X'_{m+2+k} \\
				&= P_{m+1}X'_{m+2}X'_{m+2+k} - q^{-\frac{1}{2}}P_m X'_{m+2+k} \\
				&= q^{(-1)^{k-1}} P_{m+1}X'_{m+2+k}X'_{m+2} - q^{-\frac{1}{2}}q^{(-1)^{k+1}}X'_{m+2+k}P_m \\
				&= q^{(-1)^{k-1}} X'_{m+2+k}P_{m+1}X'_{m+2} - q^{-\frac{1}{2}}q^{(-1)^{k-1}}X'_{m+2+k}P_m \\
				&=q^{(-1)^{k-1}}X'_{m+2+k}P_{m+2}.
		\end{align*}
		Similarly, if $m+2$ is even, then
		\begin{align*}
			P_{m+2}X'_{m+2+k}=X'_{m+2+k}P_{m+2},
		\end{align*}
		which proves the induction step.
	\end{proof}

	\begin{lem}
		For every $m\geq 1$ and $1 \leq i \leq n-m$, the following holds~:
		\begin{equation*}
			P_{m+1,i}=\begin{cases}
				X'_{m+1} P_{m,i} - \qp P_{m-1,i}& \textrm{if } m \textrm{ is even}~;\\
			\qp(X'_{m+1} P_{m,i}- P_{m-1,i}) & \textrm{if } m \textrm{ is odd.}
			\end{cases}
		\end{equation*}
	\end{lem}
	\begin{proof}
		Again, we use the notation $P_m(X'_1,\cdots,X'_m)=P_m$ since it is enough to prove the result for $i=1$. If $m$ is even, we have
		\begin{align*}
		P_mX'_{m+1} - q^{-\frac{1}{2}}P_{m-1} &= q^{-\frac{1}{2}}(P_{m-1}X'_m - P_{m-2})X'_{m+1} - q^{-\frac{1}{2}}P_{m-1} \\
						&= q^{-\frac{1}{2}}(P_{m-1}X'_m X'_{m+1} - P_{m-2}X'_{m+1} - P_{m-1}) \\
						&= q^{-\frac{1}{2}}(P_{m-1}(X'_m X'_{m+1}-1) - X'_{m+1}P_{m-2}) \\
						&= q^{-\frac{1}{2}}(qP_{m-1}(X'_{m+1} X'_m -1) - X'_{m+1}P_{m-2}) \\
						&= q^{-\frac{1}{2}}(qP_{m-1}X'_{m+1} X'_m -qP_{m-1} - X'_{m+1}P_{m-2}). \\
		\end{align*}
		Using Lemma \ref{lem:pnxn}, we thus get
		\begin{align*}
		P_mX'_{m+1} - q^{-\frac{1}{2}}P_{m-1} 
						&= q^{-\frac{1}{2}}(X'_{m+1}P_{m-1}X'_m -qP_{m-1} - X'_{m+1}P_{m-2}) \\
						&= X'_{m+1}q^{-\frac{1}{2}}(P_{m-1}X'_m - P_{m-2}) - \qp P_{m-1} \\
						&= X'_{m+1}P_m - \qp P_{m-1}.
		\end{align*}
		Similarly, if $m$ is odd,
		\begin{align*}
			q^{-\frac{1}{2}}(P_mX'_{m+1} - P_{m-1}) = \qp(X'_{m+1}P_m -P_{m-1}).
		\end{align*}
	\end{proof}

	We can now prove the key relation for these polynomials~:
	\begin{prop}\label{prop:frisepn}
		The $P_m$ satisfy the quantum frieze relation~:
		$$P_{m,i}P_{m,i+1} = \qp P_{m+1,i}P_{m-1,i+1}+1$$
		for any $1 \leq i \leq n-m$.
	\end{prop}
	\begin{proof}
		Without loss of generality, we assume that $i=1$. Assume that the relation holds for a given $P_m$. If $m$ is odd, then 
		\begin{align*}
		&P_{m+1}(X'_1,\cdots,X'_{m+1})P_{m+1}(X'_2,\cdots,X'_{m+2}) \\
		&= P_{m+1}(X'_1,\cdots,X'_{m+1})\qp(X'_{m+2}P_{m}(X'_2,\cdots,X'_{m+1})-P_{m-1}(X'_2,\cdots,X'_{m})) \\
		&= \qp P_{m+1}(X'_1,\cdots,X'_{m+1})X'_{m+2}P_{m}(X'_2,\cdots,X'_{m+1})\\
		&\quad-\qp P_{m+1}(X'_1,\cdots,X'_{m+1})P_{m-1}(X'_2,\cdots,X'_{m}) \\
		&= \qp(P_{m+2}(X'_1,\cdots,X'_{m+2})+q^{-\frac{1}{2}}P_{m}(X'_1,\cdots,X'_m))P_{m}(X'_2,\cdots,X'_{m+1})\\
		&\quad-\qp P_{m+1}(X'_1,\cdots,X'_{m+1})P_{m-1}(X'_2,\cdots,X'_{m}) \\
		&=\qp P_{m+2}(X'_1,\cdots,X'_{m+2})P_{m}(X'_2,\cdots,X'_{m+1})\\
		&\quad+ P_{m}(X'_1,\cdots,X'_m)P_{m}(X'_2,\cdots,X'_{m+1})\\
		&\quad-\qp P_{m+1}(X'_1,\cdots,X'_{m+1})P_{m-1}(X'_2,\cdots,X'_{m}) \\
		&=\qp P_{m+2}(X'_1,\cdots,X'_{m+2})P_{m}(X'_2,\cdots,X'_{m+1}) + 1. 
		\end{align*}
		If $m$ is even, we similarly obtain
		\begin{align*}
		P_{m+1}(X'_1,\cdots,&X'_{m+1})P_{m+1}(X'_2,\cdots,X'_{m+2})\\
		&\quad =\qp P_{m+2}(X'_1,\cdots,X'_{m+2})P_{m}(X'_2,\cdots,X'_{m+1}) + 1 .
		\end{align*}
	\end{proof}

	We can now finish the proof of Theorem \ref{thm:qCheb}.
	\begin{proof}[Proof of Theorem \ref{thm:qCheb}]
		Let $f$ be the quantum frieze of variables associated with $Q$. Mutating the initial seed in the direction $i$, it is easily seen that $$f(1,i) = X^{-e_i + e_{i+1}} + X^{-e_i + e_{i-1}} = X'_i = P_1(X'_i)$$ for any  $1\leq i\leq n$. Therefore, the theorem follows from Lemma \ref{lem:mouth}, Theorem \ref{thm:bij} and Proposition \ref{prop:frisepn}. 
	\end{proof}

	We thus obtain a very simple description of the quantum cluster variables in $\mathcal A_\Sigma$, which directly proves the lower bound phenomenon in this particular case~:
	\begin{corol}
		Let $\Sigma=(B,\Lambda,\X)$ be a quantum seed where $B$ is the incidence matrix of a linearly oriented quiver of Dynkin type $A$ (with an even number $n$ of vertices). Then
		$$\mathcal X_\Sigma = \X \sqcup \ens {P_{j}(X_i', \ldots, X_{i+j-1}')\ | \ 1 \leq i \leq n, \, j \leq n-i+1}.$$ \hfill \qed
	\end{corol}

\section*{Acknowledgements}
	This article was written while the first author was funded by the ISM for an undergraduate research training supervised by the second author and while the second author was a CRM-ISM postdoctoral fellow at the Universit\'e de Sherbrooke under the supervision of Ibrahim Assem, Thomas Br\"ustle and Virginie Charette.

\newcommand{\etalchar}[1]{$^{#1}$}
\providecommand{\bysame}{\leavevmode\hbox to3em{\hrulefill}\thinspace}
\providecommand{\MR}{\relax\ifhmode\unskip\space\fi MR }
\providecommand{\MRhref}[2]{%
  \href{http://www.ams.org/mathscinet-getitem?mr=#1}{#2}
}
\providecommand{\href}[2]{#2}


\end{document}